\documentclass[11pt,twoside]{article}

\usepackage{amssymb}
\usepackage{amsmath}

\usepackage{pgf}
\usepackage[latin1]{inputenc}
\usepackage{amsfonts}
\usepackage{amssymb}
\usepackage{amsmath,multirow}
\usepackage{amsthm}
\usepackage{epsfig,array}
\usepackage{url}
\usepackage{hyperref}
\usepackage{fancyhdr}
\usepackage{footnote}
\usepackage{algorithm}
\makeatletter
\renewenvironment{thebibliography}[1]
     {\section*{\refname}%
      \@mkboth{\MakeUppercase\refname}{\MakeUppercase\refname}%
      \list{\@biblabel{\@arabic\c@enumiv}}%
           {\settowidth\labelwidth{\@biblabel{#1}}%
            \leftmargin\labelwidth
            \advance\leftmargin\labelsep
            \@openbib@code
            \usecounter{enumiv}%
            \let\p@enumiv\@empty
            \itemsep=0pt
            \parsep=0pt
            \leftmargin=\parindent
            \itemindent=-\parindent
            \renewcommand\theenumiv{\@arabic\c@enumiv}}%
      \sloppy
      \clubpenalty4000
      \@clubpenalty \clubpenalty
      \widowpenalty4000%
      \sfcode`\.\@m}
     {\def\@noitemerr
       {\@latex@warning{Empty `thebibliography' environment}}%
      \endlist}
\makeatother

\newtheorem{definition}{ Definition}

\newtheorem{lem}{ Lemma}

\newtheorem{thm}{ Theorem}
\newtheorem{expl}{ Example}

\begin{document}





	




\title{\Large{\bf Role of New Kernel Function in Complexity Analysis of an Interior Point Algorithm for Semi definite Linear Complementarity Problem}}
\author{Nabila Abdessemed, Rachid Benacer and  Naima Boudiaf
\thanks{Nabila Abdessemed (n.abdessemed@univ-Batna2.dz), Department of Mathematics, Mostefa Ben Boula\"{i}d University (Batna2), Rachid Benacer (r.benacer@univ-batna2.dz), LTM laboratory, Department of Mathematics, Mostefa Ben Boula\"{i}d University (Batna2), and Naima Boudiaf (n.boudiaf@univ-batna2.dz), EDPA Laboratory, Department of Mathematics, Mostefa Ben Boula\"{i}d University (Batna2), Batna, Algeria.}}
\maketitle

{\bf Abstract}
{\emph{
 In this paper, we introduce a new kernel function which differs from previous functions, and play an important role for generating  a new design of primal-dual interior point algorithms for semidefinite linear complementarity problem. Its properties, allow us a great simplicity for the  analysis of interior-point method, therefore the complexity of large-update primal-dual interior point is the best so far. Numerical tests have shown that the use of this function gave a big improvement in the results concerning the time and the number of iterations. so is  well promising and perform well enough in practice in comparison with some other existing results in the literature. 
}}
\begin{flushleft}
\end{flushleft}
\section{Introduction}
Let $ S^{n} $ denotes the linear space of all $ n \times n $ real symmetric matrices, $ S_{+}^{n} $  and
$ S_{++}^{n} $ is the cone of symmetric positive semidefinite, and symmetric positive definite matrices respectively.
  The semidefinite linear complementarity problem (SDLCP) \\
\textit{ Find a pair of matrices }$  (X, Y) \in \times{S}^{n}\times{S}^{n}$  \textit {that satisfies the following conditions } \\
\begin{equation}
\label{equation: 1}
\begin{array}{lll}
 X,  Y \in S_{+}^{n},    Y= L(X)+Q,     \textit{   and   }    X \bullet Y  = Tr (XY)=0.           
\end{array}
\end{equation}
Where $ L: S^{n} \rightarrow S^{n} $ is a given linear transformation and $ X, $   $  Y, $  $  Q \in S^{n} $. 

 Interior point methods (IPMS) have been known for several decades, Since the invention of interior point methods by Karmarker, In 1984 \cite{karmarker}, whith an important contribution was made by \textbf{Nestrov} and \textbf{Todd} \cite{todd}. This methods considered the powerful tools to solve linear optimization (LO) and can be extended to more general cases such as complementarity problem (CP ), semidefinite optimization(SDO) and  semidefinite linear complementarity problem (SDLCP).\\ 
 The semidefinite linear complementarity problem (SDLCP) can be also viewed as a generalization of the standard linear complementarity problem (LCP) and included the geometric monotone semidefinite linear complementerity introduced by  \cite{kojima}, so it became the object of many studies of research these last years and have important applications in mathematical programming and various areas of engineering and scientific fields (see \cite{peyghami2},\cite{roos}). \\
 Because their polynomial complexity and their simulation efficiency, primal-dual following path are the most attractive methods among interior point to solve a large wide of optimization problems \cite{peyghami1}, \cite{peng1}, \cite{wang}, \cite{wright}. These methods are based on the kernel functions for determining new search directions and new proximity functions for analyzing the complexity of these algorithms, thus we have shown the important role of the kernel function in generating a new design of primal-dual interior point algorithm.\\ Also these methods are introduced by Bai et al \cite{bai}, and Elghami \cite{elghami 1} for (LO) and (SDO) and extended by many authors for different problems in mathematical programming \cite{achache}, \cite{boudiaf}, \cite{ elghami 2}, \cite{kheirfam }, \cite{krislock}.\\
The polynomial complexity of large update primal-dual algorithms is improved in contrast with the classical complexity given by logarithmic barrier functions by using this new form.\\
A kernel function is an univariate strictly convex function which is defined for all positive real $ t $ and is minimal at $  t=1$, whereas the minimal value equals $ 0 $. In the other words $ \psi (t) $ is a kernel function when it is twice differentiable and satisfies the following conditions
  \begin{equation*}
   \psi (1) =\psi'(V)= 0,  \psi''(t) > 0   \   for \ all \  t>0    \\\    and \
 \lim_{t \rightarrow 0} \psi(t)= \lim_{t \rightarrow +\infty} \psi(t)= +\infty. 
\end{equation*}
We can describe by its second derivative, as follows 
\begin{equation}
\label{equation: 23}
\psi(t)= \int_1^t \int_1^\zeta \psi''(\xi) d\xi d\zeta.
\end{equation}
This function may be extended to a scaled barrier function $ \Psi $ defined from $ S^{n} $ to $ S^{n} $ by $ \Psi(V)=Tr(\psi(V)) $ where $ V $ is a symmetric positive definite matrix.\\
In this paper, we establish the polynomial complexity for (SDLCP) by introducing the following new kernel function
\begin{equation}
\label{equation: 2}
\psi (t) = \frac{1}{2} (2t^{2}+ \frac{1}{t^{2}} -5)+ e^{\frac{1}{t}-1}.
\end{equation}
\\
The goal of this paper is to investigate such a new kernel function and the corresponding barrier function and show that our large-update primal-dual algorithm has favourable complexity bound in terms of elegant analytic properties of this kernel function. We show that the (SDLCP) is generalization of (SDO), we loose the orthogonality of the search direction matrices. Therefore, the analysis of search direction and step size is a little different from (SDO) case, this will be studied in detail later.\\
The paper is organized as follows. First in Sec. 2, we present the generic primal-dual algorithm, based on Nestrov-Todd direction. and the new kernel function and its growth properties for (SDLCP) are presented in Sect. 3, in Sect. 4, we derive the complexity results for the algorithm (an estimation of the step size and its default value, the worst case iteration complexity). In Sec. 5, some numerical results are provided.  Finally, a conclusion in Sec. 6.\\
Throughout the paper we use the following notation and we review some known facts about matrices and matrix functions which will be used in the analysis of the algorithm.\\
The expression $ X \succeq 0 $ ( $ X \succ 0 $) means that  $ X \in S_{+}^{n} $ ( $ X \in S_{++}^{n}  $).
The trace of $ n\times n $ matrix $ X $ is denoted by $  Tr(X) = \sum \limits_{i=1}^{n} x_{ii}.$
 The Frobenius norm of a matrix $ X \in \mathbb{R}^{n\times n}$ is defined by $  \Vert X \Vert_{F} = \sqrt{X \bullet X} = \sqrt{Tr(X^{T}X)}. $ 
 For any $ X \succ0  $, $ \lambda_{i}(X) $, $ 1 \leq i \leq n $, denote its eigenvalues.  
  $ Q^{1/2} $ denotes the symmetric square root, for any $ Q \in S_{++}^{n} $, .
The identity matrix of order $ n $ is denoted by $ I $. The diagonal  with the vector $x$ is given by $ X = diag(x). $ 
we denote by $ \lambda(V) $ the vector of eigenvalues of $ V \in S_{++}^{n} $,  arranged in
non-increasing order, that is $ \lambda_{1}(V) \geq \lambda_{2}(V) \geq \ldots \geq \lambda_{n}(V). $
For two real valued functions $ f(x), g(x): \mathbb{R}_{+}^{n} \rightarrow \mathbb{R}_{++}^{n} $, $ f(x)= O (g(x)) $ if $ f(x) \leq k g(x)$ and $  f(x)= \Theta (g(x)) $  if  $k_{1}g(x) \leq f(x) \leq k_{2} g(x) $ for some positive constants $  k $, $ k_{1} $ and $ k_{2} $.
\begin{thm}
(Spectral theorem for symmetric matrices \cite{bai})
The real $ n \times n$ martix $ A $ is symmetric if and only if there exists a matrix $ Q \in \mathbb{R}^{n \times n} $ such that $ Q^{T}Q = I $ and $ Q^{T}AQ = B $, where $ I $ is the $ n \times n $ identity matrix and $ B $ is a diagonal matrix.
\end{thm}
\begin{definition}(\cite{elghami 1}, Definition 3.2.1)
Let $V$ be a symetric matrix, and let
\begin{equation}
\label{equation: 3}
V = Q^{T} diag(\lambda_{1} (V), \lambda2 (V), \ldots , \lambda_{n} (V)) Q
\end{equation}
where $Q$ is any orthogonal matrix that diagonalizes $ V $, and let $ \psi(t) $ be defined as in Eq. (\ref{equation: 2}). The matrix valued function $ \psi: S^{n} \rightarrow S^{n}$ is defined
\begin{equation}
\label{equation: 4}
\psi(V) =  Q^{T} diag(\psi( \lambda_{1} (V)  ), \psi( \lambda_{2} (V)  ), \ldots , \psi( \lambda_{n} (V)  )) Q.
\end{equation}
\end{definition}
Note that $ \psi(V) $ depends only on the restriction of $ \psi(t) $ to the set of eigenvalues of V. Since $ \psi(t) $ is differentiable, the derivative $ \psi'(t) $ is well defined for $ t > 0 $. Hence, replacing $ \psi( \lambda_{i} (V)) $ in Eq. (\ref{equation: 4}) by $ \psi'( \lambda_{i} (V)) $, we obtain that the matrix function $ \psi'(V) $ is defined as well. Using $ \psi $, we define the barrier function (or proximity function) $ \Psi(V): S_{+}^{n} \rightarrow \mathbb{R}_{+} $ as follows
\begin{equation}
\label{equation: 5}
	\Psi(V) = Tr(\psi(V)) = \sum \limits_{i=1}^{n} \psi(\lambda_{i}(V)).
\end{equation}
When we use the function $ \psi(.) $ and its first three derivatives $ \psi'(.) $, $ \psi''(.) $ and $ \psi'''(.) $
without any specification, it denotes a matrix function if the argument is a matrix and a univariate function (from $ \mathbb{R} $ to $ \mathbb{R} $ ) if the argument is in $ \mathbb{R} $.\\
In \cite{elghami 2},\cite{horn}, we can be found some concepts related to matrix functions.
\section{Presentation of Problem}
The feasible set, the strict feasible set and the solution set of (\ref{equation: 1})
are subsets of $ \mathbb{R}^{n \times n} $ denoted respectively by 
\begin{equation*}
\begin{array}{ccl}
 \mathcal{F}&=& \lbrace (X, Y) \in S^{n}  \times  S^{n}, Y- L(X) = Q: X \succeq 0, Y \succeq 0  \rbrace,\\
\mathcal{F}^{0}&=& \lbrace (X, Y) \in \mathcal{F}: X \succ 0, Y \succ 0  \rbrace,\\
\mathcal{S}&=& \lbrace (X, Y) \in \mathcal{F}: Tr(X Y) = 0  \rbrace.
\end{array}
\end{equation*}
The set $ \mathcal{S} $ is nonempty and compact, if $ \mathcal{F}^{0} $ is not empty and $ L $ is monotone.
As we know, the basic idea of primal-dual IPMs is to relax the third equation (complementarity condition) in problem (\ref{equation: 1}) with the following parameterized system
\begin{equation}
\label{equation: 13}
\left\{
\begin{array}{lll}
XY= \mu I, \\
Y- L(X)=Q,\\
X\succ0, Y\succ0.
\end{array}
\right.
\end{equation}
where $ \mu > 0  $, and I is the identity matrix.\\
Since $ L $ is a linear monotone transformation and (SDLCP) is strictly feasible (i.e., there exists $ (X_{0}, Y_{0}) \in \mathcal{F}^{0}$), the System (\ref{equation: 13}) has a unique solution for any $ \mu > 0 $. As $ \mu \rightarrow 0 $ the sequence $ (X(\mu), Y(\mu)) $ approaches the solution $ (X,Y) $ of problem (SDLCP).\\
Suppose that the point $  (X,Y)  $ is strictly feasible. The natural way to define a search direction is to follow the Newton approach to linearize the first equation in System (\ref{equation: 13}) by replacing $ X $ and $  Y$ with $ X_{+} = X + \Delta X $ and $ Y_{+}
= Y + \Delta Y $, respectively. This leads to the following system
\begin{equation}
\label{equation: 14}
\left\{
\begin{array}{lll}
X \Delta Y + \Delta XY  = \mu I- XY, \\
 L(\Delta X)=\Delta Y.
\end{array}
\right.
\end{equation}
Or equivalently
\begin{equation}
\label{equation: 15}
\left\{
\begin{array}{lll}
\Delta X+X\Delta YY^{-1}  = \mu Y^{-1}- X, \\
 L(\Delta X)=\Delta Y.
\end{array}
\right.
\end{equation}
 In general case, the Newton system has a unique solution not necessarily symmetric, because $ \Delta X $ is not symmetric due to the matrix $  X\Delta YY^{-1}$. Many researchers have proposed methods for symmetrizing the first equation in System (\ref{equation: 15} ) by using an invertible matrix $ P $ and the term $  X\Delta YY^{-1}$ is replaced by $ P\Delta YP^{T} $. Thus, we obtain
 \begin{equation}
\label{equation: 16}
\left\{
\begin{array}{lll}
\Delta X+P\Delta YP^{T}  = \mu Y^{-1}- X, \\
 L(\Delta X)=\Delta Y.
\end{array}
\right.
\end{equation}
 
In \cite{todd}, Todd studied several symmetrization schema. Among them, we consider the Nesterov-Todd(NT) symmetrization schema where  $  P$ is defined as 
\begin{equation*} 
\begin{array}{ccl}
P &=& X^{^{\frac{1}{2}}}(X^{\frac{1}{2}} Y X^{\frac{1}{2}})^{-\frac{1}{2}}  X^{\frac{1}{2}}\\
& =& Y^{-\frac{1}{2}} (Y^{\frac{1}{2}} XY^{\frac{1}{2}} )^{\frac{1}{2}} Y^{-\frac{1}{2}},
\end{array}
\end{equation*}
Let $ D = P^{\frac{1}{2}} $ where  $ P^{\frac{1}{2}} $ denotes the symmetric square root of $ P. $  \\
The matrix  $ D $  can be used to scale  $ X$  and  $ Y$  to the same matrix $ V $ defined by
 \begin {equation}
\label{equation: 17}
V= \frac{1}{\sqrt{\mu}}D^{-1}XD^{-1}= \frac{1}{\sqrt{\mu}}DYD,
\end{equation}
thus we have
\begin{equation}
\label{equation: 18}
V^{2}= \frac{1}{\mu}D^{-1}XYD.
\end{equation}
Note that the matrix $  V  $ and  $D $ are symmetric and positive definite. Let us further define
\begin{equation}
\label{equation: 19}
\begin{array}{ccl}
D_{X}&=& \frac{1}{\sqrt{\mu}}D^{-1} \Delta X D^{-1},\\
D_{Y}&=&\frac{1}{\sqrt{\mu}}D \Delta Y D.
\end{array}
\end{equation}
So by using Eqs.(\ref{equation: 17}) and (\ref{equation: 19}), the System (\ref{equation: 16}) becomes
\begin{equation}
\label{equation: 20}
\left\{
\begin{array}{lll}
 D_{X}+D_{Y} =D_{V},\\
\tilde{L}(D_{X})=D_{Y}.
\end{array}
\right.
\end{equation}
Where $  \tilde{L}(D_{X})= D L(D D_{X} D )D $    and      $ D_{V}=V^{^{-1}}-V.  $\\

The linear transformation $ \tilde{L} $ is also monotone on $S^{n}$. Under our hypothesis the new linear System (\ref{equation: 20}) has a unique symmetric solution $ (D_{X},D_{Y}) $. These directions are not orthogonal, because
\[ 
\begin{array}{ccl}
D_{X} \bullet D_{Y}&=&Tr(D_{Y}D_{X})\\
 &=&Tr(D_{X}D_{Y})\\
&=&(\frac{1}{\sqrt{\mu}}D^{-1} \Delta X D^{-1})\bullet (\frac{1}{\sqrt{\mu}}D \Delta Y D)\\
&=& \frac{1}{\mu}\Delta X  \bullet \Delta Y \\
&=& \frac{1}{\mu}\Delta X  \bullet L(\Delta X )\geq 0.
\end{array}
\] 
thus, it is only difference between SDO and SDLCP problem.\\
So far, we have described the schema that defines classical NT-direction. Now following \cite{elghami 1, peyghami2, peyghami1} and \cite{wang}, we replace the right hand side of the first equation in System(\ref{equation: 20}) by $ -\psi'(V) $. Thus we will use the following system to define new search directions 
\begin{equation}
\label{equation: 21}
\left\{
\begin{array}{lll}
 D_{X}+D_{Y} = -\psi'(V),\\
\tilde{L}(D_{X})=D_{Y}.
\end{array}
\right.
\end{equation}
The new search directions $ D_{X} $ and $ D_{Y} $ are obtained by solving System (\ref{equation: 21}), so that $ \Delta{X} $ and $ \Delta{Y} $ are computed via Eq.(\ref{equation: 19}). By taking along the search directions with a step size $ \alpha $ defined by some line search rules, we can construct a new couple $ (X_{+},Y_{+}) $ according to $ X_{+} = X + \alpha \Delta X $ and $ Y_{+} = Y + \alpha \Delta Y $.
\begin{algorithm} 
\renewcommand{\thealgorithm }{}
\caption{Algorithm 1  Generic interior point algorithm for SDLCP}
\textbf{Input:} A threshold parameter $ \tau \geq 1 $;
\ \  \  \   \  \   \  \  an accuracy parameter $ \epsilon \geq 0 $;
\ \  \  \   \  \   \  \ barrier update parameter $ \theta $, $0 < \theta< 1$;
\ \  \  \   \  \   \  \ $ X^{0} \succ0$, $ Y^{0}\succ0 $ and $ \mu^{0}=Tr(X^{0}Y^{0})/n $  such that  $ \Psi(X^{0}, Y^{0}, \mu^{0})\leq \tau $\\
\textbf{begin}\\
\ \  \  \   \  \   \  \ $ X := X^{0}; Y := Y^{0}; \mu := \mu^{0}; $\\
\  \  \ \   \textbf{while} $ n \mu \geq \epsilon $ \textbf{do}\\
  \  \ \   \  \ \ \textbf{begin}\\
  \  \ \   \  \ \ \  \ \   \  \ \ $ \mu= (1-\theta)\mu ;$\\
  \  \  \ \  \  \  \ \  \  \  \ \   \textbf{while} $ \Psi(X, Y, \mu)>\tau $ \textbf{do}\\
 \  \ \   \  \ \  \  \  \ \  \  \  \ \  \textbf{begin}\\
    \  \ \   \  \ \  \  \ \   \  \ \ \  \  \  \ \  Solve System (\ref{equation: 21} ) and use Eq. (\ref{equation: 19}) to obtain $ (\Delta X, \Delta Y) ;$\\
      \  \ \   \  \ \  \  \ \   \  \ \  \  \  \ \  determine a suitable step size $ \alpha $;\\
\  \ \   \  \ \  \  \ \   \  \ \  \  \  \ \        update $ (X, Y) := (X, Y) + \alpha(\Delta X, \Delta Y) $\\
\  \ \   \  \ \  \  \ \   \  \ \  \  \  \ \ \textbf{end}\\
\  \ \   \  \ \  \  \ \    \textbf{end}\\

\ \   \  \ \   \textbf{end} 
\caption{1  Generic interior point algorithm for SDLCP}
\end{algorithm}
\newpage
	\section{Properties of New Kernel Function}
In this part, we present the new kernel function and we give those properties that are crucial in our complexity analysis.\\
Let
\begin{equation}
\label{equation: 22}
\psi (t) = \frac{1}{2} (2t^{2}+ \frac{1}{t^{2}} -5)+ e^{\frac{1}{t}-1}.
\end{equation}
 We list the first three derivatives of $ \psi $ as below
\begin{equation}
\label{equation: 24}
\psi'(t)= 2t- \frac{1}{t^{3}}-\frac{1}{t^{2}}e^{\frac{1}{t}-1},
\end{equation}
\begin{equation}
\label{equation: 25}
\psi''(t)= 2+ \frac{3}{t^{4}}+(\frac{2}{t^{3}}+ \frac{1}{t^{4}})e^{\frac{1}{t}-1} > 1,
\end{equation}
\begin{equation}
\label{equation: 26}
\psi'''(t)= - \frac{12}{t^{5}} - (\frac{6}{t^{5}}+\frac{1}{t^{6}}+ \frac{6}{t^{4}})e^{\frac{1}{t}-1} <0.
\end{equation}
\begin{lem}
Let $ \psi(t)$ be as defined in Eq.(\ref{equation: 22}). Then
\begin{equation}
\label{equation: 27}
t \psi''(t)+\psi'(t)>0, t<1, 
\end{equation}
\begin{equation}
\label{equation: 28}
 \psi'''(t)<0,  t>0, 
\end{equation}
\begin{equation}
\label{equation: 29}
t \psi''(t)-\psi'(t)>0, t>1, 
\end{equation}
\begin{equation}
\label{equation: 30}
2\psi''(t)^{2}-\psi'''(t)\psi'(t)>0, t<1, 
\end{equation}
\begin{equation}
\label{equation: 31}
\psi''(t)\psi'(\beta t)- \beta \psi'(t)\psi''(\beta t)>0, t>1, \beta > 1. 
\end{equation}
\end{lem}
\begin{proof}
For the first inequality, using Eqs.(\ref{equation: 24}) and (\ref{equation: 25}), it follows that 
\[
\begin{array}{ccl}
t\psi''(t)+ \psi'(t)&=& 4t+ \frac{2}{t^{3}}+(\frac{1}{t^{2}}+ \frac{1}{t^{3}})e^{\frac{1}{t}-1}>0    \textit{ if } t<1.
\end{array} 
\] 
The second inequality $ \psi'''(t)<0$, for all  $ t>0 $ it is clear from Eq.(\ref{equation: 26}).
\[
\begin{array}{ccl}
t\psi''(t)- \psi'(t) = \frac{4}{t^{3}}+(\frac{1}{t^{3}}+\frac{3}{t^{2}})e^{\frac{1}{t}-1}>0  \textit{ if } t>1.\\
2\psi''(t)^{2}-\psi'''(t)\psi'(t) = K(t)+H(t)e^{2(\frac{1}{t}-1)}+Q(t)e^{(\frac{1}{t}-1)}>0 \textit{ if } t<1. \\
\end{array} 
\] 
\[
\begin{array}{ccl}
Where, K(t)=8+\frac{48}{t^{4}}+\frac{6}{t^{8}},\ H(t)=(\frac{2}{t^{6}}+\frac{2}{t^{7}}+ \frac{1}{t^{8}}),\ and \\
Q(t)=(\frac{28}{t^{3}}+\frac{20}{t^{4}}+\frac{2}{t^{5}}+\frac{6}{t^{7}}+\frac{6}{t^{8}}-\frac{1}{t^{9}}) \\
\end{array} 
\] 
Finally, to obtain Eq. (\ref{equation: 31}), using Eqs. (\ref{equation: 28})  and  (\ref{equation: 29}).
\end{proof}
Now let as define the proximity measure $ \delta(V) $ as follows
\begin{equation}
\label{equation: 32}
\begin{array}{ccl}
\delta(V)&=&\frac{1}{2} \Vert  -\psi'(V) \Vert\\
&=& \frac{1}{2} \sqrt{Tr(\psi'(V)^{2}})\\
&=& \frac{1}{2} \Vert  D_{X} +D_{Y} \Vert.
\end{array}  
\end{equation}
 Note that $ \delta(V) = 0    \Leftrightarrow    V = I  \Leftrightarrow   \Psi(V) = 0 .$
\begin{thm} \cite{peng1}
\label{thm4}
Suppose that $ V_{1} $ and $ V_{2} $ are symmetric positive definite and $ \Psi $ is the real valued matrix function induced by the matrix function $ \psi. $ Then, 
\[
\Psi\left( \left[( V_{1}^\frac{1}{2}V_{2}V_{1}^\frac{1}{2})^\frac{1}{2} \right] \right) \leq \frac{1}{2} ( \Psi(V_{1}) +\Psi (V_{2})).
\]
\end{thm}
\begin{lem}
\label{lem2}
For $\psi(t)$, we have the following 
\begin{description}
\item[1.] $ \psi(t) $ is exponential convex, for all $ t > 0 $,
\item[2.] $ \frac{1}{{2}}(t-1)^{2}\leq \psi(t)\leq \frac{1}{{2}}(\psi'(t))^{2} $, $ t>0 $,
\item[3.] $\psi(t) \leq 4(t-1)^{2}  $, $ t\geq1. $
\end{description}
\end{lem}
Now, let $ \varrho: [ 0, \infty) \rightarrow  [ 1, \infty)  $ be the inverse function of $ \psi(t) $, for all $ t \geq 1 $ then we have the following lemma
\begin{lem}
\label{lem3}
For $ \psi(t) $, we have
\[
\sqrt{1+s} \leq \varrho (s) \leq 1+ \sqrt{2s}, \textit{  } s \geq 0.
\]
\end{lem}
\begin{thm} \cite{elghami 1}
\label{thm5}
Let $ \varrho: [ 0, \infty) \rightarrow  [ 1, \infty)  $ be the inverse function of $ \psi(t) $, $ t \geq 1 $. Then we have 
\[
\Psi(\beta V) \leq n \psi(\beta \varrho( \frac{\Psi(V)}{n})), \\\ \beta \geq 1   \\\  for  \\\\  V \in S_{++}^{n}.
\] 
\end{thm}
\begin{thm}
Let $ 0 \leq \theta\leq 1$ and $ V_{+}= \frac{V}{\sqrt{1-\theta}}$. If $ \Psi(V) \leq \tau $, then we have
\[
\Psi(V_{+}) \leq \frac{2}{1-\theta}(\sqrt{2\tau}+\sqrt{n} \theta)^{2}.
\] 
\end{thm}
\begin{proof}
 Using Theorem (\ref{thm5}) with $ \beta= \frac{1}{\sqrt{1-\theta}} $, Lemmas (\ref{lem2}), (\ref{lem3}) and $ \Psi(V)\leq \tau $, we have 
\[ 
\begin{array}{ccl}
\Psi (V_{+) }& \leq & n\psi\left( \frac{1}{\sqrt{1-\theta}} \varrho ( \frac{\Psi(V)}{n}\right)    
  \leq \frac{4n}{2}\left( \frac{\varrho ( \frac{\Psi(V)}{n})}{\sqrt{1-\theta}} -1\right) ^{2}= 2n\left( \frac{\varrho ( \frac{\Psi(V)}{n})-(\sqrt{1-\theta})}{\sqrt{1-\theta}} \right) ^{2}\leq 2n\left( \frac{1+\sqrt{2(\frac{\Psi(V)}{n}})-\sqrt{1-\theta}}{\sqrt{1-\theta}} \right) ^{2}\\\\
 &\leq & 2n\left( \frac{1+\sqrt{2(\frac{\tau}{n}})-\sqrt{1-\theta}}{\sqrt{1-\theta}} \right) ^{2}
\leq  2n\left( \frac{\sqrt{(\frac{2\tau}{n}})+\theta}{\sqrt{1-\theta}} \right) ^{2}
    \leq \frac{2}{1-\theta}\left( \sqrt{2\tau}+\sqrt{n}\theta \right) ^{2},    
\end{array}   
\]
 Where the last inequality is holds since $ 1-\sqrt{1-\theta}  =  \frac{\theta}{1+\sqrt{1-\theta}} \leq \theta,$   for all $   0\leq \theta < 1. $\\
Denote $ \Psi_{0}=\frac{2}{1-\theta}\left( \sqrt{2\tau} + \sqrt{n}\theta\right) ^{2}$. Then $ \Psi_{0} $ is an upper bound of $\Psi(V)  $ during the process of the algorithm.
\end{proof}
\section{Complexity Analysis}
\subsection{An estimation of the step size}
  The aim of this paper is to define a new kernel function, and to obtain new complexity results for an (SDLCP) problem, during an inner iteration, we compute a default step size  $ \alpha $, and the decrease of the proximity function. \\
After an inner iteration, new iterates $  X_{+} = X + \alpha \Delta X =\sqrt{\mu}D(V+\alpha D_{X})D$and $Y_{+} = Y + \alpha \Delta Y =\sqrt{\mu}D^{-1}(V+\alpha D_{Y})D^{-1}  $
 are generated, where $ \alpha $ is the step size and $  D_{X}$, $ D_{Y} $ and $ D $ are defined by (\ref{equation: 19}).
On the other hand, from (\ref{equation: 17}), we have $  V_{+}^{2} =(V +\alpha D_{X})(V+\alpha D_{Y})$ and it is clear that the matrix $V_{+}^{2}$ is similar to the matrix $ (V +\alpha D_{X})^{\frac{1}{2}}(V+\alpha D_{Y})(V +\alpha D_{X})^{\frac{1}{2}}. $
By assuming that $ (V +\alpha D_{X})\succ 0$ and $ (V +\alpha D_{Y})\succ  0 $ for such feasible step size $ \alpha $ and we deduce that they have the same eigenvalues.
Since the proximity after one step is defined by :
\begin{center}
$ \Psi(V_{+})=\Psi( [(V +\alpha D_{X})^{\frac{1}{2}}(V+\alpha D_{Y})(V +\alpha D_{X})^\frac{1}{2}]^\frac{1}{2}) $
\end{center}	
By Theorem (\ref{thm4}), we have $ \Psi(V_{+})\leq\frac{1}{2}[\Psi( (V +\alpha D_{X})+\Psi(V+\alpha D_{Y})]$. 
Define, for $\alpha > 0 $ , $ f(\alpha)=\Psi(V_{+})-\Psi(V) $ and $  f_{1}(\alpha)=\frac{1}{2}[\Psi( (V +\alpha D_{X})+\Psi(V+\alpha D_{Y})]-\Psi(V)$\\
Then $  f(\alpha)$ is the difference of the proximity between a new iterate and a current iterate for a fixed $\mu>0$.
It is easily seen that, $ f_{1}(0)=f(0)=0 $ and $  f(\alpha)\leq f_{1}(\alpha)$.
 Furthermore, $ f_{1}(\alpha) $ is a convex function.
\\Now, to estimate the decrease of the proximity during one step, we need the two successive derivatives of $ f_{1}(\alpha) $ with respect to $\alpha $.
\\By using the rule of differentiability \cite{horn}, \cite{peng1}, we obtain
\begin{center}
$ f'_{1}(\alpha)= \frac{1}{2}Tr(\psi'((V +\alpha D_{X})D_{X}+\psi'(V+\alpha D_{Y})D_{Y}) $ 
\end{center}
 and  
\begin{center}
 $ f''_{1}(\alpha)= \frac{1}{2}Tr(\psi''((V +\alpha D_{X})D_{X}^{2}+\psi''(V+\alpha D_{Y})D_{Y}^{2}) $
\end{center}
Hence, by using (\ref{equation: 19}) and (\ref{equation: 32}), we obtain 
\begin{center}
$ f'_{1}(0)= \frac{1}{2}Tr(\psi'((V)(D_{X}+D_{Y}))=\frac{1}{2}Tr(-\psi'(V)^{2})=-2\delta^{2}(V). $ 
\end{center}
In what follows, we use the short notation $ \delta(V):=\delta $
\begin{lem}\cite[Lemma 3.4.4]{elghami 1}
One has 
\begin{center}
 $ f''_{1}(\alpha)\leq 2\delta^{2}\psi''(\lambda_{n}(V)-2 \alpha\delta)  $
\end{center}
\end{lem}
\begin{lem}\cite[Lemma 3.4.5]{elghami 1}
If the step size $\alpha $ satisfies 
\begin{center}
 $-\psi'(\lambda_{n}(V)-2 \alpha\delta)+\psi'(\lambda_{n}(V)\leq 2\delta.  $
 \begin{flushleft}
 One has $ f'_{1}(\alpha)\leq 0.$
 \end{flushleft} 
\end{center}
\end{lem}
\begin{lem}\cite[Lemma 3.4.6]{elghami 1}
\label{lem6}
Let $  \rho  : [0,\infty)\rightarrow (0, 1]$ denote the inverse function of the restriction of $ -\frac{1}{2}\psi'(t) $ on the interval (0,1], then the largest possible value of the step size of $ \alpha $ satisfying $ f'_{1}(\alpha)\leq 0 $ is given by
\begin{center}
$ \overline{\alpha }=\frac{1}{2\delta }(\rho (\delta )-\rho (2\delta )). $
\end{center}
\end{lem}
\begin{lem}
\label{lem7}
Let $ \rho  $ and  $ \overline{\alpha } $ the same as be defined in Lemma (\ref{lem6}). Then
\begin{center}
$ \overline{\alpha }\geqslant\tilde {\alpha}=\frac{1}{\psi''(\rho(2\delta))}.  $
\end{center}
\end{lem}
We need to compute $ \rho (2\delta )=s $; where $ \rho  : [0,\infty)\rightarrow (0, 1]$ be the inverse of $ -\frac{1}{2}\psi'(t) $ for all $ t\in [0,1)$. This implies 
\[ 
\begin{array}{ccl}
 -\psi'(t)=4\delta
&\Leftrightarrow & -2t +\frac{1}{t^{3}}+\frac{1}{t^{2}}e^{\frac{1}{t} -1}=4\delta\\
& \Leftrightarrow & e^{\frac{1}{t} -1}= t^{2}(4\delta+2t-\frac{1}{t^{3}})   \ (*)\\
&\Rightarrow& t\geq\frac{1}{1+log(4\delta+1)} 
\end{array} 
\] 
Using the definition of  $ \psi''(t) $ and (*). If $ s\leq 1,$   we have   $ \frac{s-1}{s^{5}}\leq 0 $   and  $ \frac{1}{s^{2}} \leq (1+log(4\delta+1))^{2}. $
\begin{equation}
\psi''(t)= 2+ \frac{3}{t^{4}}+(\frac{2}{t^{3}}+ \frac{1}{t^{4}})e^{\frac{1}{t}-1}\leq 6 + 2(6\delta + 1)(1 + log(4\delta + 1))^{2}
\end{equation}
\begin{equation}
\tilde {\alpha}:=\frac{1}{\psi''(\rho(2\delta)}=\frac{1}{\psi''(s)}\geq\frac{1}{6 + 2(6\delta + 1)(1 + log(4\delta + 1))^{2}.}
\end{equation}
Next lemma shows that the proximity function $ \psi(t) $  with the default step size $ \alpha $ is decreasing
\begin{lem}\cite[Lemma 3.4]{elghami 2}
\label{lem8}
Let $ h(t)$ be a twice differentiable convex function with $ h(0)=0, h'(0)<0$, and let $ h(t)$ attains its (global) minimum at $ t > 0 $. If $ h''(t) $ is increasing for $  t \in [0, t^{\ast}]$, then $ h(t)=\frac{th'(0)}{2}. $
 $ f_{1}(\alpha) $ holds the condition of the above lemma,
 for all $ 0\leq \alpha\leq\overline{\alpha }; $
\begin{equation}
 f(\alpha) \leq f_{1}(\alpha)\leq \frac{f'_{1}(0)}{2}\alpha      
\end{equation}
\end{lem}
Then we have the following lemmas to obtain the upper bound for the decreasing value of the proximity in the inner  iteration.
\begin{lem}
\label{lem9}
For any $ \alpha $ satisfying $ \alpha\leq\overline{\alpha } $, we have : $ f(\alpha) \leq-\alpha\delta^{2}$
\end{lem}
\begin{lem}
\label{lem10}
Let $ \Psi(V)\geq1 $; $ \rho $ and $ \tilde{\alpha} $ be defined as in Lemma (\ref{lem6}) and Lemma(\ref{lem7}). Then, one has:
\begin{equation}
 f(\tilde{\alpha}) \leq -\frac{\delta^{2}}{\psi''(\rho(2\delta)}).      
\end{equation}
\end{lem}
\begin{proof}\
Lemma (\ref{lem9}) and the fact that $ \overline{\alpha }\geq \tilde{\alpha}$, imply that
\[ 
\begin{array}{ccl}
f(\tilde{\alpha}) \leq -\tilde{\alpha}\delta^{2}
 = -\frac{\delta^{2}}{\psi''(\rho(2\delta)}
  \leq -\frac{\delta^{2}}{6 + 2(6\delta + 1)(1 + log(4\delta + 1))^{2}}
  \leq -\frac{\delta^{2}}{6 + 2(6\delta + \sqrt{2}\delta)(1 + log(4\delta + 1))^{2}}
 \leq -\frac{\delta^{2}}{6 + 2\delta(6 + \sqrt{2})(1 + log(4\delta + 1))^{2}}\\\\
  \leq -\frac{1}{2}\left( \frac{\Psi}{6 + 2\frac{\sqrt{\Psi}}{\sqrt{2}}(6 + \sqrt{2})(1 + log(4\frac{\sqrt{\Psi}}{\sqrt{2}} + 1))^{2}}\right) \leq -\frac{\sqrt{\Psi}}{(16 + 12\sqrt{2})(1 + log(2\sqrt{2\Psi_{0} + 1))^{2}}}
  \leq -\frac{\sqrt{\Psi_{0}}}{33(1 + log(2\sqrt{2\Psi_{0} + 1))^{2}}};          
\end{array}   
\] 
$ (16 + 12\sqrt{2})\simeq 33 $,       \      $ f(\tilde{\alpha}) \leq  -\frac{\sqrt{\Psi_{0}}}{33(1 + log(2\sqrt{2\Psi_{0} + 1))^{2}}}. $
\end{proof}.     
 \subsection{Iteration bound}
To come back to the situation where $ \Psi(V)\leq\tau $ after $ \mu-$update, we have to count how many inner iterations.\
Let the value of $ \Psi(V) $ after $ \mu-$update be denoted by $ \Psi_{0} $ and the
subsequent values by $ \Psi_{k} $, for $ k = 0, 1,\ldots ,K-1 $, where $ K $ is the total number of inner iterations per the outer iteration. Then we have
\begin{equation}
\Psi_{k-1}>\tau, 0 \leq\Psi_{k}\leq\tau 
\end{equation}
\begin{lem} \cite{peng2}
\label{lem11}
 Let $ t_{0}, t_{1}, \ldots , t_{k}$ be a sequence of positive numbers such that
\begin{center}
$ t_{k+1}\leq t_{k}-\beta t_{k}^{1-\gamma} , \textit{  } k = 0, 1, \ldots, K-1   $
\end{center}
Where   $ \beta>0  $ and $ 0<\gamma\leq1 $. Then 
\begin{center}
$ K\leq\lceil\dfrac{t_{0}^{\gamma}}{\beta\gamma}\rceil  $
\end{center}
\end{lem}
Letting  $ t_{k}=\Psi_{k}$, \  $\beta=\frac{1}{33(1 + log(2\sqrt{2\Psi_{0}} + 1))^{2}}$,  \  $  \gamma=\frac{1}{2}$
\begin{lem}
Let $ K $ be the total number of inner iterations in the outer iteration. Then we have
\begin{center}
$ K \leq 66 \left(  1 + log(2\sqrt{2\Psi_{0}} + 1))^{2}\right)\Psi_{0}^{\frac{1}{2}}.$ 
\end{center}
\end{lem}
\begin{proof}
Using Lemma (\ref{lem11}), we get the result.
\end{proof}
Now, we estimate the total number of iterations of our algorithm.
\begin{thm}
If $ \tau\geq1 $, the total number of iterations is not more than
\begin{center}
$ 66(1 + log(2\sqrt{2\Psi_{0}}+1))^{2})\Psi_{0}^{\frac{1}{2}}  {\dfrac{1}{ \theta}   log\dfrac{n\mu^{0}}{\epsilon}} . $
\end{center}
\end{thm}
\begin{proof}
In the algorithm,$ n\mu\leq\epsilon, $   $ \mu^{k}=(1-\theta)^{k}\mu^{0}$ and  $\mu^{0}=\dfrac{x_{0}^{t}y_{0}}{n}.$
By simple computation, we have
\begin{center}
$ k\leq{\dfrac{1}{ \theta} log\frac{n\mu^{0}}{\epsilon}}. $
\end{center}
By multiplying the number of outer iterations and the number of inner
iterations, we get an upper bound for the total number of iterations, namely
\begin{center}
$ {\dfrac{K}{\theta}}log{\frac{n\mu^{0}}{\epsilon}}\leq {\dfrac{66}{\theta}(1 + log(2\sqrt{2\Psi_{0}}+1))^{2})\Psi_{0}^{\frac{1}{2}}log\dfrac{n\mu^{0}}{\epsilon}}. $
\end{center}
This completes the proof.
\end{proof}
we assume that $ \tau=O(n) $,   $ \theta=\Theta(1) $ and $  \Psi_{0}^{\frac{1}{2}}=O(\sqrt{n}.)$\\
The algorithm will obtain the solution of the problem at most $ O(\sqrt{n}(logn)^{2}log\dfrac{n}{\epsilon})$.
 \section{Numerical results}
 The main purpose of this section is to present three monotone SDLCPs for testing the effectiveness of algorithm. The implementation is manipulated in "Matlab". Here we use "inn", "out" and "T" which means the inner,  outer iterations number and the time produced by the algorithm 1, respectively.  The choice of different values of the parameters shows their effect on reducing the number of iterations. \\
 
 In all experiments,  we use $ \tau=2 $ , $ \epsilon=10^{-6} $,  $  \theta \in\lbrace   0.5,   0.89,  0.95   \rbrace  $, and     $ \alpha \in\lbrace  0.3,  0.5,  0.7,  0.8,  0.9, 1,  1/log(4\delta),  1/1+log(4\delta+1)   \rbrace $ ,  the theoretical barrier parameter $ \mu_{0}\in\lbrace  Tr(XY)/n,   0.05,   0.005,   0.0005 \rbrace $. We provide a feasible initial point $ (X_{0},Y_{0})$ such that IPC and $\Psi(X_{0},Y_{0},\mu_{0})\leqslant\tau $ are satisfied.\\
 
 The first example is the monotone SDLCP defined by two sided multiplicative linear transformation \cite{achache} . The second is monotone SDLCP which is equivalent to the symmetric semidefinite least squares(SDLS)problem  and the third one is reformulated from nonsymmetric semidefinite least squares(NS-SDLS)problem \cite{krislock}, in the second and third example, $L$ is Lyaponov linear transformation. \\  
 \newpage
 \begin{expl}
 The data of the monotone SDLCP is given by  $  L(X)=AXA^{T}, $\\
   where \\
\[
A=\left( 
\begin{array}{ccccccccl}
17.25 & -1.75 & -1.75 & -1.75 & -1.75 \\ 
-1.75 & 16.25 & -2 & 0  & 0 \\ 
-1.75 & -2 &16.25  & -2 &0  \\ 
-1.75 & 0 &-2 & 16.25 &-2  \\ 
-1.75 & 0 &  0& -2 &16.25 &\\
\end{array}
\right) 
\]
 \\
and 
\[
Q= \left( 
\begin{array}{cccccl }
-9.25 & 1.25 & 1.25 & 1.25 & 1.25 \\ 
1.25 & -8.25 & 1.5 &0  & 0 \\ 
1.25 & 1.5 &-8.25  & 1.5 &0  \\ 
1.25  & 0 &1.5 & -8.25 &1.5  \\ 
1.25 & 0 &  0& 1.5 &-8.25 &\\
\end{array}
\right) 
\]\\
The strictly feasible initial starting point $ X^{0}\succ 0 $ is given by $ X^{0}=Diag(0.0620, \ldots ,0.0620). $\\ 
The unique solution $X^{*} \in S_{+}^{5}  $  is given by \\
\[
X^{*}=\left( 
\begin{array}{ccccc}
0.0313 & 0.0020 & 0.0020 & 0.0020 & 0.0020\\ 
0.0020 & 0.0313 & 0.0019 &0  & 0 \\ 
0.0020& 0.0019 &0.0312  & 0.0019 & 0  \\ 
0.0020 & 0 &0.0019 & 0.0312 &0.0019  \\ 
0.0020 & 0 &  0& 0.0019&0.0313 \\
\end{array}
\right) 
 \]\\
The number of inner, outer iterations and the time for several choices of $ \alpha $, $\theta $ and  $\mu $ obtained by algorithm 1 are presented in the following tables
\begin{center}
\begin{table} [!h]
\begin{center}
\begin{tabular}{|c|c|c|c|c|}
\hline
                    $ \theta=0.5 $ \\
                                    
\hline
$  \alpha /\mu$ &$ Tr(XY)/n $ & $ 0.05$ & $ 0.005$  & $ 0.0005$\\
\hline
    &  inn / out / T  &  inn / out / T  &  inn /out / T  & inn /  out / T \\
 \hline   
0.3 & 74 / 22 / 0.17 & 71 / 19 /0.17  & 65 / 15 / 0.15 &  61 / 12 / 0.15  \\
\hline 
0.5 & 43 / 22 / 0.12 &  41 /19 /0.12  &  37 / 15 / 0.12 & 35 / 12 /0.10  \\
\hline 
0.7 & 23 / 22 / 0.09 &  22 / 19 / 0.17  &  20 / 15 / 0.07 &  19 / 12 / 0.07 \\
\hline 
0.9 & 22/  22 / 0.10 &  21 / 19 / 0.15 &  18 / 15 / 0.07  &  16 / 12 / 0.09\\
\hline
1 & 22 / 22 / 0.10 &  19 / 18 / 0.09  &  17 / 15 /0.07  &  15 / 12 / 0.07 \\
\hline
1/log(4$\delta$)& 7 /  22 / 0.07 & 16 / 19 / 0.09 &  12 / 15 / 0.09 &  25 / 12 /0.10 \\
\hline 
1/log(4$\delta$+1)&   20 / 22  / 0.09 & 21 / 19 / 0.06 & 24 / 15 / 0.09 &  30 / 12 /0.10 \\
\hline                                  
\end{tabular}
\end{center}
\begin{center}
\caption{Number of inner, outer iterations and the time for several choices of $\alpha $  and  $\mu $with $\theta=0.5 $ }
\end{center}
\end{table}
\end{center}
\newpage 
\begin{center}
\begin{table} [!h]
\begin{center}
\begin{tabular}{|c|c|c|c|c|}
\hline
                    $ \theta=0.89 $ \\
                                    
\hline
$  \alpha /\mu$ &$ Tr(XY)/n  $ & $ 0.05$ & $ 0.005$  & $ 0.0005$\\
\hline
    &  inn / out / T  &  inn / out / T  &  inn /out / T  & inn /  out / T \\
 \hline   
0.3 & 58 / 7 / 0.14 & 57 / 6 /0.14  & 55 / 5 / 0.12 &  53 / 4 / 0.14  \\
\hline 
0.5 & 31 / 7 / 0.10 &  31 /6 /0.10  &  30 / 5 / 0.09 & 28 / 4 /0.09  \\
\hline 
0.7 & 21 / 7 / 0.09 &  20 / 6 / 0.07  &  19 / 5 / 0.09 &  18 / 4 / 0.09 \\
\hline 
0.9 & 14/  7 / 0.07 &  14 / 6 / 0.07 &  13 / 5 / 0.06  &  12 / 4 / 0.06\\
\hline
1 & 9 / 7 / 0.06 &  8 / 6 / 0.06  &  8 / 5 /0.06  &  7 / 4 / 0.06 \\
\hline
1/log(4$\delta$)& 14 /  7 / 0.07 & 20 / 6 / 0.07 &  22 / 5 / 0.09 &  36 / 4 /0.10 \\
\hline 
1/log(4$\delta$+1)&   22 / 7  / 0.07 & 24 / 6 / 0.09 & 30 / 5 / 0.09 &  36 / 4 /0.10 \\
\hline                                    
\end{tabular}
\end{center}
\begin{center}
\caption{Number of inner, outer iterations and the time for several choices of $ \alpha $ and $\mu $with $\theta=0.89 $ }
\end{center}
\end{table}
\end{center}
\begin{center}
\begin{table} [!h]
\begin{center}
\begin{tabular}{|c|c|c|c|c|}
\hline
                    $ \theta=0.95 $ \\                                 
\hline
$  \alpha /\mu$ &$ Tr(XY)/n  $ & $ 0.05$ & $ 0.005$  & $ 0.0005$\\
\hline
     &  inn / out / T  &  inn / out / T  &  inn /out / T  & inn /  out / T \\
 \hline   
0.3 & 53 / 5 / 0.12 & 60 / 5 /0.14  & 56 / 4 / 0.14 &  51 / 3 / 0.10  \\
\hline 
0.5 & 29 / 5 / 0.09 &  33 /5 /0.09  &  30 / 4 / 0.09 & 28 / 3 /0.09  \\
\hline 
0.7 & 18 / 5 / 0.07 &  20 / 5 / 0.07  &  18 / 4 / 0.07 &  17 / 3 / 0.07 \\
\hline 
0.9 & 11/  5 / 0.07 &  12 / 5 / 0.07 &  11 / 4 / 0.07  &  10 / 3 / 0.06\\
\hline
1 & 7 / 5 / 0.06 &  7 / 5 / 0.06  &  7 / 4 /0.06  &  6 / 3 / 0.06 \\
\hline
1/log(4$\delta$)& 20 /  5 / 0.07 & 23 / 5 / 0.07 &  32 / 4 / 0.10 &  38 / 3 /0.10 \\
\hline 
1/log(4$\delta$+1)&   24 / 5  / 0.09 & 31 / 5 / 0.09 & 34 / 4 / 0.09 &  41 / 3 /0.10 \\
\hline                           
\end{tabular}
\end{center}
\begin{center}
\caption{Number of inner, outer iterations and the time for several choices of $ \alpha $ and $\mu $with $\theta=0.95 $ }
\end{center}
\end{table}
\end{center}
 \end{expl}
 \begin{expl}
 The data of the monotone SDLCP which is equivalent to the symmetric semidefinite least squares (SDLS) problem is given by  $ L(X)=\frac{1}{2}(A^{T}AX+XA^{T}A)   \textit{ and }   Q = -\frac{1}{2}(A^{T}B+B^{T}A).$
   Where   
\[
A=\left( 
\begin{array}{ccccccccl}
6 & -1 & 0 & 0 & 0 \\ 
-0.1 & 6 & -1 & 0 & 0 \\ 
0 & -0.1 &6  & -1 & 0  \\ 
0 & 0 &-0.1 & 6 &-1  \\ 
0 & 0 &  0& -0.1 &6\\
0 & 0 &  0& 0 &-0.1 &\\
\end{array}
\right) 
\]
 \\ 
 and
\[
B= \left( 
\begin{array}{cccccl }
 1   & 0 & 0 & 0 & 0 \\ 
-0.4 & 1 & 0 & 0  & 0 \\ 
-0.4 & -0.4  &  1  & 0 & 0  \\ 
-0.4 & 0 &-0.4 & 1 & 0  \\ 
-0.4 & 0 &  0 & -0.4 & 1 \\
-0.4 & 0 & 0 & 0& -0.4\\
\end{array}
\right) 
\]
\\

The strictly feasible initial point $ X^{0}\succ 0 $ defined by $ X^{0}=Diag(0.2369, \ldots ,0.2369). $\\
The unique solution $ X^{*} \in S_{+}^{5}  $ of the proposed example is given by\\
 
 \[
X^{*}=\left( 
\begin{array}{ccccc}
 0.1639 & -0.0215 & -0.0342 & -0.0328 & -0.0300\\ 
-0.0215 &  0.1553 & -0.0227 & -0.0019 & -0.0027 \\ 
-0.0342 & -0.0227 &  0.1558 & -0.0194 &  0.0014  \\ 
-0.0328 & -0.0019 & -0.0194 &  0.1564 & -0.0189  \\ 
-0.0300 & -0.0027 &  0.0014 & -0.0189 &  0.1598 \\
\end{array}
\right) 
 \]\\
 
The number of inner, outer iterations and the time for several choices of $ \alpha$, $\theta$ and $\mu $ are presented in the following tables.
\begin{center}
\begin{table} [!h]
\begin{center}
\begin{tabular}{|c|c|c|c|c|}
\hline
                    $ \theta=0.5 $ \\                           
\hline
$  \alpha /\mu$ &$ Tr(XY)/n $ & $ 0.05$ & $ 0.005$  & $ 0.0005$\\
\hline
     &  inn / out / T  &  inn / out / T  &  inn /out / T  & inn /  out / T \\
 \hline   
0.3 & 74 / 22 / 0.21 & 69 / 18 /0.20  & 65 / 15 / 0.07 &  61 / 12 / 0.18  \\
\hline 
0.5 & 43 / 22 / 0.15 &  40 /18 /0.15  &  37 / 15 / 0.17 & 35 / 12 /0.15  \\
\hline 
0.7 & 23 / 22 / 0.11 &  21 / 18 / 0.11  &  20 / 15 / 0.11 &  19 / 12 / 0.10 \\
\hline 
0.9 & 22/  22 / 0.10 &  20 / 18 / 0.12 &  18 / 15 / 0.09  &  16 / 12 / 0.10\\
\hline
1 & 22 / 22 / 0.12 &  19 / 18 / 0.12  &  17 / 15 /0.11  &  15 / 12 / 0.06 \\
\hline
1/log(4$\delta$)& 8 /  22 / 0.07 & 12 / 18 / 0.11 &  18 / 15 / 0.12 &  26 / 12 /0.14 \\
\hline 
1/log(4$\delta$+1)&   20 / 22  / 0.10 & 21 / 18 / 0.10 & 25 / 15 / 0.10 &  31 / 12 /0.12 \\
\hline                                  
\end{tabular}
\end{center}
\begin{center}
\caption{Number of inner, outer iterations and the time for several choices of $ \alpha$, and $\mu$ with $\theta=0.5 $. }
\end{center}
\end{table}
\end{center}
\begin{center}
\begin{table} [!h]
\begin{center}
\begin{tabular}{|c|c|c|c|c|}
\hline
                    $ \theta=0.89 $ \\                            
\hline
$  \alpha /\mu$ &$ Tr(XY)/n $ & $ 0.05$ & $ 0.005$  & $ 0.0005$\\
\hline
    &  inn / out / T  &  inn / out / T  &  inn /out / T  & inn /  out / T \\
 \hline   
0.3 & 58 / 7 / 0.20 & 57 / 6 /0.20  & 55 / 5 / 0.17 &  53 / 4 / 0.17  \\
\hline 
0.5 & 31 / 7 / 0.12 &  31 /6 /0.12  &  29 / 5 / 0.10 & 28 / 4 /0.12  \\
\hline 
0.7 & 21 / 7 / 0.10 &  20 / 6 / 0.10  &  19 / 5 / 0.07 &  18 / 4 / 0.10 \\
\hline 
0.9 & 14/  7 / 0.09 &  14 / 6 / 0.09 &  13 / 5 / 0.07  &  12 / 4 / 0.07\\
\hline
1 & 8 / 7 / 0.04 &  8 / 6 / 0.06  &  8 / 5 /0.06  &  8 / 4 / 0.07 \\
\hline
1/log(4$\delta$)& 19 /  7 / 0.07 & 23 / 6 / 0.10 &  29 / 5 / 0.12 &  31 / 4 /0.14 \\
\hline 
1/log(4$\delta$+1)&   23 / 7  / 0.10 & 26 / 6 / 0.10 & 31 / 5 / 0.14 &  38 / 4 /0.15 \\
\hline                               
\end{tabular}
\end{center}
\begin{center}
\caption{Number of inner, outer iterations and the time for several choices of $ \alpha$, and $\mu $with $\theta=0.89 $.}
\end{center}
\end{table}
\end{center} 
\begin{center}
\begin{table} [!h]
\begin{center}
\begin{tabular}{|c|c|c|c|c|}
\hline
                    $ \theta=0.95 $ \\                                  
\hline
$  \alpha /\mu$ &$ Tr(XY)/n $ & $ 0.05$ & $ 0.005$  & $ 0.0005$\\
\hline
     &  inn / out / T  &  inn / out / T  &  inn /out / T  & inn /  out / T \\
 \hline   
0.3 & 54 / 5 / 0.18 & 61 / 5 /0.20  & 56 / 4 / 0.20&  52 / 3 / 0.18  \\
\hline 
0.5 & 29 / 5 / 0.12 &  33 /5 /0.10  &  30 / 4 / 0.12 & 28 / 3 /0.09  \\
\hline 
0.7 &18 / 5 / 0.09 &  20 / 5 / 0.06  &  18 / 4 / 0.07 &  17 / 3 / 0.09 \\
\hline 
0.9 & 11/  5 / 0.07 &  12 / 5 / 0.07 &  11 / 4 / 0.07  &  10 / 3 / 0.06\\
\hline
1 & 7 / 5 / 0.07 &  7 / 5 / 0.04  &  7 / 4 /0.06  &  7 / 3 / 0.07 \\
\hline
1/log(4$\delta$)& 18 /  5 / 0.07 & 21 / 5 / 0.07 &  25 / 4 / 0.12 &  41 / 3 /0.17 \\
\hline 
1/log(4$\delta$+1)&   25 / 5  / 0.10 & 32 / 5 / 0.12 & 36 / 4 / 0.12 &  42 / 3 /0.14 \\
\hline                                
\end{tabular}
\end{center}
\begin{center}
\caption{Number of inner, outer iterations and the time for several choices of $\alpha$, and $\mu$ with $\theta=0.95 $. }
\end{center}
\end{table}
\end{center} 
\end{expl}
\newpage
 \begin{expl}
 
 We consider the monotone SDLCP which is reformulated from NS-SDLS problem:\\
 The matrices $ A $ and $ B $ of  NS-SDLS are given by\\
 
\[
A=\left( 
\begin{array}{ccccccccl}
-0.3157 &  0.0330  &  0.0603 \\ 
-0.3274 & -0.0158  &  0.0625 \\ 
-0.3569 &  0.0787  &  0.0563   \\ 
-0.2994 &  0.0301  &  0.0496  \\ 
-0.3243 & -0.0048  &  0.0715 \\
-0.3447 &  0.0736  &  0.0545 \\
-0.2417 &  0.0709  &  0.0522  \\
-0.2063 & -0.0099  &  0.0233  \\
-0.3285 &  0.1585  &  0.0979  \\
-0.2484 &  0.0878  &  0.0622  \\
-0.2196 &  0.0023  &  0.0280  \\
-0.3148 &  0.1506  &  0.0922 \\
\end{array}
\right) 
\]\\
and
 \\
 \[
B=\left( 
\begin{array}{ccccccccl}
-1.4257 &  0.1528  &  0.4398 \\ 
-1.4024 & -0.3092  &  0.4187 \\ 
-1.3766 &  0.4366  &  0.4197 \\ 
-1.4274 &  0.1424  &  0.4353 \\ 
-1.3994 & -0.3095  &  0.4206 \\
-1.3716 &  0.4285  &  0.4193 \\
-1.4269 &  0.1581  &  0.4335 \\
-1.4015 &  0.3229  &  0.4214 \\
-1.3767 & -0.4189  &  0.4333 \\
-1.4257 &  0.1515  &  0.4358 \\
-1.3989 &  0.3276  &  0.4217 \\
-1.3724 &  0.1454  &  0.4356 \\
\end{array}
\right) 
\]
\\
Lyapunov linear transformation $ L(X) $ is symmetric and strictly monotone given by \\

$ L(X)=\frac{1}{2}(G^{-1}X+XG^{-1})$   \  and \  $ Q = -\frac{1}{2}(G^{-1}A^{T}B+B^{T}AG^{-1}).$\\
   
   Where  $ G = A^{T}A  $. The unique solution $ X^{*} $ is given by
   \\
   
  \[
X^{*} = \left( 
\begin{array}{ccccccccl}
 5.1595 &  0.3075  &  2.3185 \\ 
-0.8348 &  6.2621  &  -8.1377 \\ 
 1.5400 & -0.0070  &  0.6169 \\ 
\end{array}
\right) 
\]
\\

The number of inner, outer iterations and the time for several choices of $ \alpha$, $\mu $ and $ \theta \in \lbrace
   0.50,   0.80,   0.90 \rbrace$ with feasible starting point $ X^{0}= I $  are presented in the following tables.
\begin{center}
\begin{table} [!h]
\begin{center}
\begin{tabular}{|c|c|c|c|c|}
\hline
                    $ \theta=0.50 $ \\                               
\hline
$  \alpha /\mu$ &$ Tr(XY)/n $ & $ 0.05$ & $ 0.005$  & $ 0.0005$\\
\hline
     &  inn / out / T  &  inn / out / T  &  inn /out / T  & inn /  out / T \\
 \hline   
0.3 & 6 / 22 / 0.07 &  6/18/0.07 & 6 / 14 /0.06& 6/11/0.06   \\
\hline 
0.5 & 3 / 22 / 0.06 &  3/18/0.06  & 3 /14 /0.06  & 3/11/0.06  \\
\hline 
0.7 & 2 / 22 / 0.07 &  2/18/0.06  & 2 / 14 / 0.06  &2/11/0.04   \\
\hline 
0.9 & 2/  22 / 0.06 & 2/18/0.04  &  2 / 14 / 0.06  & 2/11/0.04 \\
\hline
1 & 2 / 22 / 0.06 &  2/18/0.06  & 2 / 14 / 0.06   &  2/11/0.06 \\
\hline
1/log(4$\delta$)& 15 /  22 / 0.09 & 17/ 18 / 0.09 & 15/ 14 /0.06 & 15 / 11 / 0.06   \\
\hline 
1/log(4$\delta$+1)&   15 / 22  / 0.07 & 17 / 18 / 0.09 & 15/ 14 /0.09 & 15 / 11 / 0.09   \\
\hline                                   
\end{tabular}
\end{center}
\begin{center}
\caption{Number of inner, outer iterations and the time for several choices of $\alpha$, and $\mu$ with $\theta=0.5 $. }
\end{center}
\end{table}
\end{center}
\newpage
\begin{center}
\begin{table} [!h]
\begin{center}
\begin{tabular}{|c|c|c|c|c|}
\hline
                    $ \theta=0.80 $ \\                                  
\hline
$  \alpha /\mu$ &$ Tr(XY)/n $ & $ 0.05$ & $ 0.005$  & $ 0.0005$\\
\hline
      &  inn / out / T  &  inn / out / T  &  inn /out / T  & inn /  out / T \\
 \hline   
0.3 & 6 / 10 / 0.04 &  6 / 8 / 0.06 & 6 / 6 /0.04 & 6 / 5 / 0.06   \\
\hline 
0.5 & 3 / 10 / 0.04 &  3 / 8 / 0.04  & 3 / 6 /0.06  & 3 / 5 / 0.06  \\
\hline 
0.7 & 2 / 10 / 0.04 &  2 / 8 / 0.06  & 2 / 6 / 0.06  & 2 / 5 / 0.04   \\
\hline 
0.9 & 2/  10 / 0.04 & 2 / 8 / 0.06  &  2 / 6 / 0.04  & 2/ 5 /0.06 \\
\hline
1 & 2 / 10 / 0.06 &  2 / 8 / 0.06  & 2 / 6 / 0.06   &  2/ 5 /0.04 \\
\hline
1/log(4$\delta$)& 16 /  10 / 0.07 & 17 / 8 / 0.09 & 15 /6 / 0.07 & 17 / 5 / 0.09    \\
\hline 
1/log(4$\delta$+1)&   16 / 10  / 0.04 & 17 / 8 / 0.07 & 16/ 6 / 0.07 & 17 / 5 / 0.07    \\
\hline                                  
\end{tabular}
\end{center}
\begin{center}
\caption{Number of inner, outer iterations and the time for several choices of $ \alpha$, and $\mu$ with$\theta=0.80 $.}
\end{center}
\end{table}
\end{center}
\newpage 
\begin{center}
\begin{table} [!h]
\begin{center}
\begin{tabular}{|c|c|c|c|c|}
\hline
                    $ \theta=0.90 $ \\                            
\hline
$  \alpha /\mu$ &$ Tr(XY)/n $ & $ 0.05$ & $ 0.005$  & $ 0.0005$\\
\hline
     &  inn / out / T  &  inn / out / T  &  inn /out / T  & inn /  out / T \\
 \hline   
0.3 & 6 / 7 / 0.06 &  6/ 6 /0.04 & 6 / 5 /0.06 & 6/ 4 / 0.06   \\
\hline 
0.5 & 3 / 7 / 0.04 &  3/ 6 /0.04  & 3 /5 /0.04  & 3/ 4 / 0.04  \\
\hline 
0.7 & 2 / 7 / 0.04 &  2/ 6 /0.06  & 2 / 5 / 0.04  & 2 / 4 / 0.06   \\
\hline 
0.9 & 2/  7 / 0.06 & 2/ 6 /0.06  &  2 / 5 / 0.06  & 2 / 4 /0.04 \\
\hline
1 & 2 / 7 / 0.04 &  2/ 6 /0.04  & 2 / 5 / 0.06   &  2 / 4 / 0.06 \\
\hline
1/log(4$\delta$)& 16 / 7 / 0.07 & 17/ 6 /0.07 & 17/5/0.07 & 17 / 4 / 0.07    \\
\hline 
1/log(4$\delta$+1)&   16 / 7  / 0.06 & 17/ 6 /0.07 & 17/5/0.07 & 17 / 4 / 0.06    \\
\hline                            
\end{tabular}
\end{center}
\begin{center}
\caption{Number of inner, outer iterations and the time for several choices of $ \alpha$, and $\mu$ with $\theta=0.90$.}
\end{center}
\end{table}
\end{center}
 \end{expl}
The results in these tables show that the algorithm based on our kernel function $ \psi(t) $ is effective, and it's number of iterations depends on the values of the parameters $ \alpha $, $ \theta $ and $ \mu $.
For all possible combinations of this parameters in practical computation, we obtained the better results than those by recent kernel functions.
\section{Conclusion} 
In this paper, we introduced a new kernel function, which contributed well to creating a new design for primal-dual interior-point algorithms. We show that the iteration bound of large-update interior point method is  $ O(\sqrt{n}(logn)^{2}log\dfrac{n}{\epsilon})$, which improves the best iteration complexity. Finally, the numerical results obtained are excellent,which indicated that our kernel function used in algorithm 1 is efficient.  


\begin{thebibliography}{99}

\bibitem{achache}  Achache. M, Tabchouche.N,  \textit{A full Nesterov-Todd step primal-dual path- folowing interior point algorithm for semidefinite linear complementarity problems}. Croation Operational Research Review Crorr 9(2018), pp 37-50, .
\bibitem{bai} Bai. Y. Q, Elghami. M  and  Roos. C, \textit{A comparative study of kernel functions for primal-dual interior point algorithms in linear optimization}, SIAM J. optim. 15(1), pp,  101-128, (2004).
\bibitem{boudiaf}  Boudiaf. N, \textit{Probl\`eme de compl\'ementarit\'e lin\'eaire semi d\'efini. Etude th\'eorique et algorithmique}, Th\`ese de doctorat en sciences math\'ematiques, universit\'e de batna, Algerie (2012).
\bibitem{elghami 1}  Elghami. M, \textit{New primal-dual interior point methods based on kernel functions}, Ph. D. thesis, Delft University, Netherland, (2005).
\bibitem{elghami 2}  Elghami. M, Guennoun. Z. A , Boula. S  and  Steihaug. T, \textit{ Interior-point methods for linear optimization based on a kernel function with trigonometric barrier term}, J. comput. App, Math, 236, pp,  3613-3623,  (2012).
\bibitem{horn}  Horn. R. A and charles. R. J. \textit{ Matrix Analysis, Cambridge university press, UK, (1986).}
\bibitem{karmarker}  Karmarker. N, \textit{ A new polynomial time algorithm, for linear programming in proceedings of the 16th Annual ACM symposium on Theory of computing}, pp, 302-311, (1984).
 \bibitem{kheirfam }  Kheirfam. B, \textit{ Primal-dual interior point algorithm for semidefinite optimization based on a new kernel function with trigonometric barrier term}, 61, pp, 659-680, (2012).
\bibitem{kojima}  Kojima. M, Shindoh. M  and  Hara. S, \textit{ Interior point methods for monotone semidefinite linear complementarity in symmetric matrices}. SIAM J. Optimization, 7,  pp,  86-125, (1997).
\bibitem{krislock} Krislock. N. G. B,  \textit{ Numerical solution of semidefinite constrained least squares problems. Master of science. The university of British colombia, Canada}, (2003).
\bibitem{peyghami2}  Peyghami. M. R,  \textit{ An interior-point approach for semi definite optimisation using new proximity functions}, Asia-Pac. J. Oper. Res, 26(3), pp 365-382, (2009).
\bibitem{peyghami1}  peyghami. M. R,  Fathi Hafshejani.S  and  Chen. S, \textit{ A primal-dual interior point method for semidefinite optimization based on a class of trigonometric barrier functions},J. Oper. Res, pp 319-323, 44(2016).
\bibitem{peyghami3} peyghami. M. R,  Fathi Hafshejani.S  and  Chen. S, \textit{ Complexity of interior point methods for linear optimization based on new trigonometric kernel function}, J, comput. App. Math, 255, pp, 74-85, (2014).
\bibitem{peng1}  Peng. J, C. Roos. C and  Terlaky. T, \textit{ New class of polynomial primal-dual methods for linear and semidefinite optimization}, European J. Oper. Res. 143(2), pp, 234-256, (2002).
\bibitem{peng2}  Peng. J, C. Roos. C and  Terlaky. T, \textit{ Self regular function and new search directions for linear and semidefinite optimization}. Math. program, 93, pp, 129-171, (2002).
\bibitem{roos}  Roos. C,  Terlaky. T and Vial. J. P, \textit{ Theory and Algorithms for linear optimization An interior point Approach}, springer, New york, (2005).
\bibitem{todd} Todd. M. J,  \textit{ A study of search directions in primal-dual interior point methods for semidefinite programming}. Optim. Methods Softw, 11, pp, 1-46, (1999).
\bibitem{wang} Wang. G. Q, Bai. Y. Q and Roos. C, \textit{ Primal-dual interior point algorithm for semidefinite optimization based on a simple kernel function}, J. Math. Model Algorithms.  4, pp, 409-433, (2005).
\bibitem{wright}  Wright. S. J, \textit{ Primal-dual interior point methods}. SIAM, Philadelphia, USA, (1997).
\end{thebibliography}
\end{document}